\documentclass{elsarticle}

\usepackage{xcolor}
\usepackage{natbib}
\usepackage{graphicx}
\usepackage{caption}
\usepackage{subcaption}
\usepackage{amssymb}
\usepackage{latexsym}
\usepackage[english]{babel}
\usepackage{amsmath}
\usepackage{makeidx}
\usepackage[utf8]{inputenc}
\usepackage{verbatim}
\usepackage[pagewise]{lineno}
\usepackage{amsthm}

\usepackage{tikz}
\usetikzlibrary{patterns}
\usepackage{tikz-cd}

\usepackage{caption}
\usepackage{subcaption}
\newtheorem{theorem}{Theorem}[section]

\newtheorem{proposition}[theorem]{Proposition}

\theoremstyle{definition}
\newtheorem{definition}[theorem]{Definition}
\newtheorem{remark}[theorem]{Remark}

\newcommand{\R}{\ensuremath{\mathbb{R}}}

\newcommand{\D}{\ensuremath{\mathcal{D}}}

\newcommand{\M}{\ensuremath{\mathcal{M}}}

\newcommand{\lp}{\left(}
\newcommand{\rp}{\right)}
\newcommand{\lc}{\left\{}
\newcommand{\rc}{\right\}}

\begin{document}
	
	\begin{frontmatter}
		
		
		
	\title{{\bf Nonholonomic Newmark method}}
		
		 \author[label1]{Alexandre Anahory Simoes}
		 \author[label2]{Sebasti\'an J. Ferraro} 
		 \author[label3]{Juan Carlos Marrero}
		 \author[label1]{David Mart\'in de Diego} 
		 \address[label1]{Instituto de Ciencias Matem\'aticas,  ICMAT
		 	c/ Nicol\'as Cabrera, n$^\textrm{o}$~13-15, Campus Cantoblanco, UAM
		 	28049 Madrid, Spain (e-mail: alexandre.anahory@icmat.es and david.martin@icmat.es)}
		 \address[label2]{Instituto de Matem\'atica (INMABB) -- Departamento de Matem\'atica, Universidad Nacional del Sur (UNS) -- CONICET, Bah\'ia Blanca, Argentina (e-mail: sferraro@uns.edu.ar)}
		\address[label3]{
			ULL-CSIC Geometr\'{\i}a Diferencial y Mec\'anica Geom\'etrica,
			Departamento de Matem\'aticas, Estad{\'\i}stica e IO, Secci\'on de
			Ma\-te\-m\'a\-ti\-cas y F{\'\i}sica, Universidad de la Laguna, La Laguna,
			Tenerife, Canary Islands, Spain (email: jcmarrer@ull.edu.es)}
		
		\begin{abstract}
		Using the nonholonomic exponential map, we generalize the well-known family of Newmark methods for nonholonomic systems. We give numerical examples including a test problem where the structure of reversible integrability responsible for good energy behaviour as described in \cite{modin} is lost. We observe that the composition of two Newmark methods is able to produce good energy behaviour on this test problem.
		\end{abstract}
	

        	\end{frontmatter}

\section{Introduction}
In numerical integration, one of the most widely used methods in nonlinear structure dynamics is without any doubt the Newmark family of numerical methods \cite{newmark}. 

As far as we know, the Newmark methods have not been extended to an important class of systems: nonholonomic systems. Briefly, a nonholonomic system is a mechanical system with external constraints on the velocities whose equations are obtained using the Lagrange-d'Alembert principle  (see \cite{Bloch}). These systems are present in a great variety of engineering and robotic environments as for instance in applications to wheeled vehicles and satellite dynamics. In this paper, we will consider only the case of linear velocity constraints since this is the case in most examples, but the extension of our nonholonomic Newmark method to the case of nonlinear constraints, explicitly time-dependent systems and nonholonomic systems with external forces is completely straightforward.

The case of linear velocity constraints is  specified by a {(in general, nonintegrable)} regular distribution ${\mathcal D}$ on the configuration space $Q$, or equivalently, by a vector subbundle $\tau_{\mathcal D}: {\mathcal D}\rightarrow Q$ of the tangent bundle $TQ$ with canonical inclusion $i_{\mathcal D}: {\mathcal D}\hookrightarrow TQ$. Therefore, the admissible curves $\gamma: I\subseteq {\mathbb R}\rightarrow Q$ must verify the following constraint equation 
\begin{equation*}\label{curve}
\gamma'(t)=\frac{d\gamma}{dt}(t)\in {\mathcal D}_{\gamma(t)} \hbox{    for all   } t\in I \, . 
\end{equation*}
The case of holonomic constraints occurs when  ${\mathcal D}$ is integrable or, equivalently, involutive. Observe that in this case, all the curves through a point $q\in Q$ satisfying the constraints must lie on the maximal integral submanifold of ${\mathcal D}$ through $q$.


In this paper, we construct nonholonomic Newmark methods in the case where $Q$ is ${\mathbb R}^n$ and discuss the possibility of composing Newmark methods to obtain higher-order methods. At the end, we test them in some nonholonomic problems. According to \cite{modin}, the reason why several numerical methods produce good energy behaviour is due to the fact that they preserve reversible integrability and most nonholonomic examples are precisely reversible integrable. The perturbed pendulum-driven CVT system is an example of an unbiased nonholonomic system since it is no longer reversible. Surprisingly, one of our methods, the composition of two Newmark methods, showed nearly preservation of energy.

The paper is structured as follows. In Section 2, we give a review of the Newmark method to integrate second-order differential equations and rewrite them in terms of a discretization of the exponential map. In Section 3, we review the definition of nonholonomic mechanics and of the nonholonomic exponential map which motivates the introduction of nonholonomic Newmark methods. In Proposition 3.5, we prove that Newmark methods with $\beta=\beta'=0$ are equivalent to a DLA method and in Proposition 3.7, we obtain a numerical method from the composition of lower order Newmark methods. In Section 4, we give three examples of nonholonomic systems including the perturbed pendulum-driven CVT system and in Section 5 we present our numerical results. Finally, in Section 6, we discuss the possibility of generalizing nonholonomic Newmark methods to general manifolds and, in particular, to a Lie group  (see  \cite{Krysl}).


 \section{Newmark method for explicit second-order differential equations}
Given a  second order differential equation
$
\frac{d^2q}{dt^2}=\Gamma(t, q, \dot{q})
$
the classical Newmark method is given by
\begin{equation}\label{Newmark:method}
	\begin{split}
		&\frac{q_{k+1}-q_k}{h}  = \dot{q}_k+h\left(\frac{1}{2}-\beta \right)\Gamma(t_k, q_k, \dot{q}_k)+h\beta \Gamma(t_{k+1}, q_{k+1}, \dot{q}_{k+1})\\
		&\frac{\dot{q}_{k+1}-\dot{q}_k}{h}  = \left(1-\gamma \right)\Gamma(t_k, q_k, \dot{q}_k)+\gamma \Gamma(t_{k+1}, q_{k+1}, \dot{q}_{k+1})
	\end{split}
\end{equation}
where $\gamma$ and $\beta$ are real numbers with $0\leq \gamma\leq 1$ and $0\leq \beta\leq 1/2$.
The Newmark method is second order accurate if and only if $\gamma=1/2$, otherwise it is only consistent. 
Moreover, this family of second order methods includes the trapezoidal rule ($\beta=1/4$) and the St\"ormer's method ($\beta=0$). In the latter case, the Newmark  method is simplified as follows: 
\begin{align*}
\frac{q_{k+1}-q_k}{h}&=\dot{q}_k+\frac{h}{2}\Gamma(t_k, q_k, \dot{q}_k)\\
\frac{\dot{q}_{k+1}-\dot{q}_k}{h}&=\frac{1}{2}\Gamma(t_k, q_k, \dot{q}_k)+\frac{1}{2}\Gamma(t_{k+1}, q_{k+1}, \dot{q}_{k+1})
\end{align*}

\subsection{Newmark method for Lagrangian systems}

The Newmark method \cite{newmark} is a classical time-stepping method that is very common in structural mechanical simulations. For simplicity, we consider a typical mechanical Lagrangian $L: T{\mathbb R}^n\longrightarrow {\mathbb R}$:
\begin{equation}\label{lag}
L(q, \dot{q})=\frac{1}{2}\dot{q} M\dot{q}^T-V(q)\, ,
\end{equation}
where $(q, \dot q)\in T{\mathbb R}^n\equiv {\mathbb R}^{2n}$, $M$ is a symmetric positive definite constant $n\times n$-matrix and $V$ is a potential function.
The corresponding 
Euler-Lagrange equations are
\begin{equation}\label{eq:NewmarkL}
\ddot{q}=-M^{-1}\nabla V(q)\,,
\end{equation}
where $\nabla$ denotes the gradient of the potential function.

The Newmark methods are widely used in simulations of such mechanical systems. In fact, they can be applied in an even more general context including external forces (cf. \cite{KMOW}).  In this case, fixing parameters  $\gamma$ and $\beta$,  equations (\ref{Newmark:method}) 
determine an integrator implicitly which gives $(q_{k+1}, \dot{q}_{k+1})$ in terms of $(q_{k}, \dot{q}_{k})$ by
\begin{align}\label{eq:NewmarkMethod}
q_{k+1}&=q_k+h\dot{q}_k+\frac{h^2}{2}\left( (1-2\beta) a_{k}+2\beta a_{k+1}\right)\\
\dot{q}_{k+1}&=\dot{q}_k+h\left( (1-\gamma) a_k+\gamma a_{k+1}\right)\, ,
\end{align}
where $a_k=-M^{-1}\nabla V(q_k)$ and  $a_{k+1}=-M^{-1}\nabla V(q_{k+1})$.

In contrast with other geometric integrators for Lagrangian systems (see \cite{MW2001}), the Newmark scheme is not especially designed to be symplectic and momentum preserving, but in \cite{KMOW} the authors show that the conservation of the symplectic form and the momentum occurs in a non-obvious way. In other words, the Newmark methods preserve a non-canonical perturbed symplectic form and a non-standard momentum.

\subsection{The Newmark method and the exponential map}

Given a second order differential equation $
\frac{d^2q}{dt^2}=\Gamma(t, q, \dot{q})
$  on Q, a point $q \in Q$ and a sufficiently small positive number $h > 0$,  we can construct the  exponential map of $\Gamma$ in $q$ at time $h$, i.e., a map $\hbox{exp}_{q,h}: U\subseteq T_{q}Q \rightarrow Q$. This map is defined taking for any vector $v\in T_qQ$ the unique trajectory of the second order differential equation with this initial condition,
that is the unique curve $\gamma: I\subset {\mathbb R}\rightarrow Q$ such that $\gamma(0)=q$, $\dot{\gamma}(0)=v$ and 
$\ddot{\gamma}(t)=\Gamma(\gamma(t), \dot{\gamma}(t))$ (see \cite{MMM3} and references therein). Then we define
	\[
		\hbox{exp}_{q,h} (v)=\gamma(h)
\]
A natural idea to derive a numerical method is to consider a discretization of the exponential map $\hbox{exp}^d_{q,h}: U\subseteq T_{q}Q \rightarrow Q$ that is, an approximation of the continuous exponential map. 
If $Q$ is a vector space, a common example of a discretization is the second order Taylor polynomial
\begin{equation}
	\text{exp}_{q,h}^d (v)=q+h v+\frac{h^2}{2}\Gamma(q, v)\; .
\end{equation}

\begin{definition}
	A discretization of the exponential map of a second order differential equation is a family of maps $\hbox{exp}^d_{q,h}:T_{q}Q \rightarrow Q$ depending on a parameter $h\in (-h_{0},h_{0})$ with $h_{0}>0$ such that
$\hbox{exp}^d_{q,0}(v_q)=q$, that is, it is a  constant map and the first and second derivatives with respect to $h$ satisfy
	$$\left. \frac{d}{dh} \right|_{h=0} \hbox{exp}^d_{q,h}(v) = v, \quad \left. \frac{d^{2}}{dh^{2}} \right|_{h=0} \hbox{exp}^d_{q,h}(v) = \Gamma(q,v).$$
\end{definition}

\begin{definition}
	The discrete flow $\Phi^h_d: TQ\rightarrow TQ$, $\Phi^h_d(q_k, v_k)=(q_{k+1}, v_{k+1})$ defined implicitly by the expression
	\begin{equation}\label{exponential:method}
	\begin{cases}
	q_{k+1} = \hbox{exp}_{q_k, h}^d (v_k) \\
	q_k = \hbox{exp}^d_{q_{k+1}, -h}(v_{k+1})
	\end{cases}
	\end{equation}
	is called the \textit{exponential method}.
\end{definition}

Observe that by the implicit function theorem, $\Phi^h_{d}$ is well-defined if $T_{v}\text{exp}_{q,h}^d$ is regular at $v=v_{k+1}$ for any $h\in (-h_{0},h_{0})$. In other words,
\[
\Phi^h_d (v_k)=\left[\hbox{exp}^d_{exp_{q_k, h}^d (v_k), -h}\right]^{-1}(q_k), \; \; \mbox{ for } v_k \in T_{q_k}Q.
\]
As we will see next, this is precisely the Newmark method with $\beta=0$ and $\gamma=1/2$.

In general, we can recover any Newmark method as a map 
$\Phi^h_d: TQ\rightarrow TQ$, $\Phi^h_d(q_k, v_k)=(q_{k+1}, v_{k+1})$
using the following discretizations of the exponential map depending of a parameter $\beta$ with $0\leq \beta\leq 1/2$: 
	\begin{equation}\label{exponential:method-1}
 \hbox{exp}^{\beta}_{q_k, h} (v_k)=q_k+hv_k+\frac{h^2}{2}\left(
(1-2\beta) \Gamma (q_k, v_k)+2\beta \Gamma (q_{k+1}, v_{k+1})\right)  
\end{equation}
and the Newmark method is rewritten as
	\begin{equation}\label{exponential:method-2}
\begin{cases}
q_{k+1} = \hbox{exp}^{\beta}_{q_k, h} (v_k) \\
q_k = \hbox{exp}^{\beta'}_{q_{k+1}, -h}(v_{k+1})
\end{cases}
\end{equation}
with parameters  $0\leq \beta, \beta'\leq 1/2$. That is
\begin{equation}\label{exponential:method-2}
\begin{split}
q_{k+1}&=q_{k}+hv_{k}+\frac{h^2}{2}
(1-2\beta) \Gamma (q_k, v_k) + h^2 \beta \Gamma (q_{k+1}, v_{k+1})\\
q_k&=q_{k+1}-hv_{k+1}+\frac{h^2}{2}
(1-2\beta') \Gamma (q_{k+1}, v_{k+1}) + h^2 \beta' \Gamma (q_k, v_k)
\end{split}
\end{equation}
Observe that that these methods are equivalent to the Newmark methods with parameters  $\beta$ and  $\gamma=(1+2\beta'-2\beta)/2$ in the expression (\ref{Newmark:method}) (in fact, if in (\ref{exponential:method-2}) we put $v_k = \dot{q}_k$ and $v_{k+1} = \dot{q}_{k+1}$ then we obtain (\ref{Newmark:method})).

\begin{remark}
	The discretization of the exponential map given in Equation (\ref{exponential:method-1}) should be understood as follows. Given a discretization  
	$\Phi^h_d: TQ\rightarrow TQ$ of the flow of a  second order differential equation  then we can define the discretization of the exponential map as
\[	
	\hbox{exp}^{\beta}_{q_k, h} (v_k)=q_k+hv_k+\frac{h^2}{2}\left(
	(1-2\beta) \Gamma (q_k, v_k)+2\beta \Gamma (\Phi^h_d (q_k, v_k))\right) \; .
	\]
	Thus, it is clear that it only depends on the  variables  $(q_k, v_k)$.
\end{remark}

\section{The nonholonomic Newmark method}

\subsection{Nonholonomic mechanics}
Consider a nonholonomic system on the configuration space $Q$ determined by a  Lagrangian function $L: TQ\rightarrow {\mathbb R}$ and nonholonomic constraints which are linear in the velocities given by a nonintegrable distribution ${\mathcal D}$. In coordinates, 
\begin{equation*}\label{LC}
\mu^{a}_{i}\lp q\rp\dot q^{i}=0,\hspace{2mm} m+1\leq a\leq n\, ,
\end{equation*}
where $\mbox{rank}\lp\mathcal{D}\rp=m\leq n$. The annihilator $\mathcal{D}^{\circ}$ is locally given by
\[
\mathcal{D}^{\circ}=\lc\mu^{a}=\mu_{i}^{a}(q)\,dq^{i};\hspace{1mm} m+1\leq a\leq n\rc\, ,
\]
where the 1-forms $\mu^{a}$ are independent.

The  equations of motion are completely determined by the   La\-gran\-ge-d'Alembert principle (\cite{Bloch}). This principle states that a curve $q: I\subset \R\rightarrow  Q$  is an admissible motion of the system if
\[
\delta\mathcal{J}=\delta\int^{T}_{0}L\lp q\lp t\rp, \dot q\lp t\rp\rp dt=0\, ,
\]
for all  variations satisfying  $\delta q\lp t\rp\in\mathcal{D}_{q\lp t\rp}$, $0\leq t\leq T$, $\delta q\lp 0\rp=\delta q\lp T\rp=0$. The velocity of the curve itself must also satisfy the constraints, that is,
$\mu_{i}^{a}(q(t))\,\dot q^{i}(t)=0$. 
 From the Lagrange-d'Alembert principle, we arrive at the well-known {\bf nonholonomic equations}
\begin{subequations}\label{nonholo-equations}
	\begin{align}
	\frac{d}{dt}\lp\frac{\partial L}{\partial\dot q^{i}}\rp-\frac{\partial L}{\partial q^{i}}&=\lambda_{a}\mu^{a}_{i}\, ,\label{Con-2}\\\label{Con-3}
	\mu_{i}^{a}(q)\,\dot q^{i}&=0\, ,
	\end{align}
\end{subequations}
where $\lambda_{a}$, $m+1\leq a \leq n$, is a set of Lagrange multipliers to be determined. 
The right-hand side of Equation (\ref{Con-2}) represents the force induced by the constraints  (reaction forces), while Equation (\ref{Con-3}) gives the linear velocity constraint condition. 

If we assume that the nonholonomic system is regular (see \cite{LMdD1996}), which is guaranteed if the Hessian matrix 
\[
(W_{ij})=\left( \frac{\partial^2 L}{\partial \dot{q}^i\partial \dot{q}^j}\right)
\]
is positive (or negative) definite, then the nonholonomic equations can be characterized as the solutions of a second order differential equation $\Gamma_{nh}$ restricted to the constraint space determined by ${\mathcal D}$. We can rewrite Equation (\ref{Con-2}) as a vector field on the tangent bundle 
$\Gamma_{nh}= \Gamma_L+\lambda_{a} Z^{a}$ 
where
\begin{align*}
\Gamma_L&=\dot{q}^i\frac{\partial}{\partial q^i}
+W^{ij}\left(\frac{\partial L}{\partial q^j}-\frac{\partial^2 L}{\partial \dot{q}^j\partial q^k}\dot{q}^k\right) \frac{\partial}{\partial \dot{q}^i}\\
Z^{a}&=W^{ij}\mu^{a}_j\frac{\partial}{\partial \dot{q}^i}
\end{align*}
where $(W^{ij})$ is the inverse matrix of $(W_{ij})$ (see \cite{LMdD1996,LM94}). Moreover, the Lagrange multipliers are completely determined and are given by the expression
\begin{equation*}
	\lambda_{a} = - \mathcal{C}_{a b}\Gamma_{L}(\mu^{b}_{i}\dot{q}^{i}),
\end{equation*}
where $(\mathcal{C}_{a b})$ is the inverse matrix of $(\mathcal{C}^{a b})=(\mu^{a}_jW^{ij}\mu^{b}_i)$. This matrix is invertible if and only if the nonholonomic system $(L, \mathcal{D})$ is regular.

\subsection{The discrete constraint space for nonholonomic systems}

Given a nonholonomic system $(L,\D)$, the \textit{nonholonomic exponential map} at $q\in Q$ and at time $h>0$ is the map
\begin{equation*}
\begin{split}
\text{exp}_{q,h}^{nh}: {\mathcal U}_q\subseteq\D_{q} & \longrightarrow Q \\
v_{q} & \mapsto c_{v_{q}}^{nh}(h)
\end{split}
\end{equation*}
sending each tangent vector $v_{q}$ in the distribution to the unique nonholonomic trajectory starting at $q$ with initial velocity $v_{q}$ evaluated at time $h$ (see \cite{AMM20} for more details; see also \cite{anahory1}). 

The fact that the space of initial velocities is restricted to the subspace $\D_{q}$, implies that the set of points reached by nonholonomic trajectories starting at $q$, that is, the image of $\text{exp}_{q,h}^{nh}$ is a submanifold of $Q$. Thus, we define the \textit{exact discrete constraint space} at $q$ as
\begin{equation}
\M_{q,h}^{nh} := \text{exp}_{q,h}^{nh}(\D_{q}).
\end{equation}
We are intentionally committing a slight abuse of notation in the definition of $\M_{q,h}^{nh}$, since not all vectors in $\D_{q}$ are guaranteed to generate a nonholonomic trajectory defined up to time $h$. But {if $h$ is sufficiently small,}  we can always consider a non-empty open subset of $\D_{q}$ generating such well-defined trajectories.

Moreover, it can be proven that $\text{exp}_{q,h}^{nh}$ is a diffeomorphism from an open subset $\mathcal{U}_{q}$ in $\D_{q}$ to $\M_{q,h}^{nh}$. Thus, in particular, the dimension of $\M_{q,h}^{nh}$ is precisely $\text{rank}(\D)$ (see \cite{AMM20, anahory1}).

This observation is particularly important, since it shows that if $q_{1}$ and $q_{0}$ are two sufficiently close points connected by a nonholonomic trajectory, then $q_{1}$ is restricted to live in the submanifold $\M_{q_{0},h}^{nh}$ with strictly lower dimension than $Q$ (in fact $\dim \M_{q_{0},h}^{nh}=m$). We will take this restriction into account when constructing numerical methods for nonholonomic systems. Though this procedure mimics the exact situation, we are also introducing a new source of error in the numerical integrator, since the discrete space must be approximated.

Assume that we have a nonholonomic system given by $(L, {\mathcal D})$ with nonholonomic dynamics given by 
\begin{equation}\label{sode-nh}
\Gamma_{nh}(q,v,\lambda)=\Gamma_L(q,v)+\lambda Z(q,v)
\end{equation}
and the Lagrange multipliers are derived from the nonholonomic constraints $\dot{c}(t)\in {\mathcal D}_{c(t)}$. 

The equations of motion of a nonholonomic system are completely determined by the nonholonomic exponential map. In fact the unique solution $\gamma: I\subset {\mathbb R}\rightarrow Q$ of the constrained SODE $\Gamma_{nh}$ with  initial condition
such that $\gamma(0)=q$, $\dot{\gamma}(0)=v_q \in {\mathcal D}_q$ is characterized by
\[
\gamma(h)=\text{exp}_{q,h}^{nh}(v_q)\; .
\]
From the properties of vector field flows, in this case $\Gamma_{nh}$, we have the following compatibility conditions
\[
\text{exp}_{q, s h}^{nh}(v_q)=\text{exp}_{\tilde{q},(s-1) h}^{nh}(\tilde{v}_{\tilde{q}})
\]
where $\tilde{q}=\gamma(h)=\text{exp}_{q,h}^{nh}(v_q)$, $\tilde{v}_{\tilde{q}}=\dot{\gamma}(h)$ and $s\in [0,1]$. In particular, for $s=0$ and $s=1$, we obtain the following system of equations 
\begin{align*}
\tilde{q}&=\text{exp}_{q, h}^{nh}(v_q)\\
q&=\text{exp}_{\tilde{q}, -h}^{nh}(\tilde{v}_{\tilde{q}})
\end{align*}
Observe that the final position and velocity satisfy the constraints $\tilde{q}\in \M_{q, h}^{nh}$ and $\tilde{v}_{\tilde{q}}\in
{\mathcal D}_{\tilde{q}}$.

\subsection{The nonholonomic Newmark method}\label{NhN_section}

These last properties are precisely the constraints that we will impose in order to obtain a nonholonomic version of the Newmark method: $(q_k, v_k)\rightarrow (q_{k+1}, v_{k+1})$. In particular, we need an appropriate discretization 
$$\text{exp}_{q, h}^{d, \beta, \lambda, \lambda'}: {\mathcal D_q}\rightarrow Q$$
of the nonholonomic exponential map depending on a parameter $0\leq \beta\leq 1/2$ and Lagrange multipliers $\lambda$ and $\lambda'$  which force the final point to satisfy a discretization of the exact discrete constraint space, denoted by ${\mathcal M}_{q_k,h}^d \subseteq Q$, with $\mbox{dim}\; \mathcal{M}_{q_k,h}^d = \mbox{rank}\lp\mathcal{D}\rp$, and the final velocity to belong to ${\mathcal D}$.
More concretely, we have the following definition
\begin{eqnarray*}
&&\text{exp}_{q_k, h}^{d, \beta, \lambda, \lambda'}(v_k)=
q_{k} + h v_{k} \\&&\qquad + \frac{h^{2}}{2}\left((1-2\beta)\Gamma_{nh}(q_{k},v_{k}, \lambda_{k}) + 2\beta \Gamma_{nh}(q_{k+1},v_{k+1}, \lambda'_{k+1})\right) 
\end{eqnarray*}
where  we are denoting the second component
$\Gamma_{nh}(q,v,\lambda)=\Gamma_L(q,v)+\lambda Z(q,v)$ of the vector field $\Gamma_{nh}$ with the same letter to avoid overloading notation. Therefore, our proposal of nonholonomic Newmark method is:
\begin{definition}
	The {\bf nonholonomic Newmark method} with parameters $(\beta, \beta')$, $0\leq \beta, \beta'\leq 1/2$ is the integrator $F_h^{\beta, \beta'}: \D\rightarrow \D$  implicitly given by
	\begin{align*}
		q_{k+1}&=\text{exp}_{q_k, h}^{d, \beta, \lambda, \lambda'}(v_k)\\
		q_k&=\text{exp}_{q_{k+1}, -h}^{d, \beta', \lambda', \lambda}(v_{k+1})\\
			q_{k+1} &\in {\mathcal M}_{q_k,h}^d \\
		v_{k+1} &\in \D_{q_{k+1}},
	\end{align*}
	or
	\begin{equation*}
	\begin{cases}
	q_{k+1} = q_{k} + h v_{k} + \frac{h^{2}}{2}\left((1-2\beta)\Gamma_{nh}(q_{k},v_{k}, \lambda_{k}) + 2\beta \Gamma_{nh}(q_{k+1},v_{k+1}, \lambda'_{k+1})\right) \\
	q_k = q_{k+1} - h v_{k+1} + \frac{h^{2}}{2}\left(2\beta'\Gamma_{nh}(q_{k},v_{k}, \lambda_{k}) + (1-2\beta') \Gamma_{nh}(q_{k+1},v_{k+1}, \lambda'_{k+1})\right) \\
	q_{k+1} \in {\mathcal M}_{q_k,h}^d \\
	v_{k+1} \in \D_{q_{k+1}}.
	\end{cases}
	\end{equation*}
\end{definition}

If the constraint distribution is given as the zero set of the functions $\phi^{a}:TQ \rightarrow \R$, i.e.,
\begin{equation*}
\phi^{a}(q_{k},v_{k})=0
\end{equation*}
and the discrete constraint space is given as the zero set of the functions $\Phi^{a}:Q\times Q \rightarrow \R$, i.e.,
\begin{equation*}
\Phi^{a}(q_{k},q_{k+1})=0,
\end{equation*}
then the discrete equations can be written as
\begin{equation*}
\begin{cases}
q_{k+1} = q_{k} + h v_{k} + \frac{h^{2}}{2}\left((1-2\beta)\Gamma_{nh}(q_{k},v_{k}, \lambda_{k}) + 2\beta \Gamma_{nh}(q_{k+1},v_{k+1}, \lambda'_{k+1})\right) \\
q_k = q_{k+1} - h v_{k+1} + \frac{h^{2}}{2}\left(2\beta'\Gamma_{nh}(q_{k},v_{k}, \lambda_{k}) + (1-2\beta') \Gamma_{nh}(q_{k+1},v_{k+1}, \lambda'_{k+1})\right) \\
\Phi^{a}(q_{k},q_{k+1})=0 \\
\phi^{a}(q_{k+1},v_{k+1})=0.
\end{cases}
\end{equation*}


\begin{remark}
In the case of holonomic constraints, that is, when the distribution $\mathcal{D}$ is integrable, the exact discrete constraint space $\mathcal{M}^{nh}_{q_k,h}$ is precisely the leaf ${\mathcal L}_{q_k}$ of the foliation by the point $q_k$ integrating the distribution and the constraint distribution is just the tangent space to each leaf (see \cite{AMM20, anahory1}). Therefore, the nonholonomic Newmark method in the holonomic case becomes (see \cite{zamm}):

 	\begin{equation*}
 	\begin{cases}
 	q_{k+1} = q_{k} + h v_{k} + \frac{h^{2}}{2}\left((1-2\beta)\Gamma_{nh}(q_{k},v_{k}, \lambda_{k}) + 2\beta \Gamma_{nh}(q_{k+1},v_{k+1}, \lambda'_{k+1})\right) \\
 	q_k = q_{k+1} - h v_{k+1} + \frac{h^{2}}{2}\left(2\beta'\Gamma_{nh}(q_{k},v_{k}, \lambda_{k}) + (1-2\beta') \Gamma_{nh}(q_{k+1},v_{k+1}, \lambda'_{k+1})\right) \\
 	q_{k+1} \in {\mathcal L}_{q_k} \\
 	v_{k+1} \in \D_{q_{k+1}}.
 	\end{cases}
 	\end{equation*}
	
\end{remark}

\begin{remark}
    A very important caveat is that when $\beta+\beta'=1/2$, the system of equations given by the nonholonomic Newmark method becomes ill-conditioned, at least for the case of mechanical Lagrangians. This is because the Jacobian matrix of the system with respect to the unknowns $(q_{k+1},v_{k+1},\lambda_k,\lambda'_{k+1})$ has two columns, those corresponding to the Lagrange multipliers, that are almost proportional. Each numerical step gives results with a large uncertainty, which accumulates rapidly. Therefore, this choice of parameters, which of course includes the case $\beta=\beta'=1/4$, should be avoided.
\end{remark}

\subsection{Discretizations of the exact discrete constraint space}\label{F_betabetaalpha_section}

Suppose that the nonholonomic constraints defining the distribution $\mathcal{D}$, as a submanifold of $TQ$, are
$$\phi^{a}(q,v)=\langle \mu^{a}(q),v \rangle $$
and, additionally, that the discrete constraints are obtained from the continuous ones in the following way
\begin{equation}\label{alpha:discretization}
	\Phi^{a}(q_{k},q_{k+1}) = \left\langle \mu^{a}\left((1-\alpha) q_{k} + \alpha q_{k+1}\right), \frac{q_{k+1}-q_{k}}{h} \right\rangle, \quad \alpha \in [0, 1].
\end{equation}
Alternatively, it would be also possible to consider
\begin{equation}\label{alpha:discretization-1}
\tilde{\Phi}^{a}(q_{k},q_{k+1}) = \left\langle (1-\alpha)\mu^{a}\left(q_{k} \right) + \alpha\mu^{a}\left(q_{k+1} \right), \frac{q_{k+1}-q_{k}}{h} \right\rangle, \quad \alpha \in [0, 1].
\end{equation}
Whenever it is clear which of the constraint discretizations we are using, we will simply denote the associated nonholonomic Newmark flow by $F_{h}^{\beta,\beta',\alpha}:\D \rightarrow \D$.

In this sense, let ${\mathcal M}_{q_k,h}^d \subseteq Q$ for each $q_{k} \in Q$ be the submanifold
$${\mathcal M}_{q_k,h}^d = \{q_{k+1} \ | \ \Phi^{a}(q_{k},q_{k+1}) = 0\}.$$

For deriving nonholonomic Newmark methods of order two, it would be interesting to assume the  following {\bf symmetry condition} in the discretization of the exact discrete constraint space: 
\[
q_{k+1}\in {\mathcal M}_{q_k,h}^d \Leftrightarrow q_{k}\in {\mathcal M}_{q_{k+1},h}^d
\]
For instance, this condition is satisfied  if $\alpha=1/2$ in the discretizations given by (\ref{alpha:discretization}) and  (\ref{alpha:discretization-1}).

As a direct consequence using the symmetry of the method  we obtain the following proposition (see \cite{hairer}). 
\begin{proposition}
	The nonholonomic Newmark methods with $\beta=\beta'$ and a symmetric discretization of the constraints are at least of order 2.
\end{proposition}

Thus, the nonholonomic Newmark method $F_{h}^{\beta,\beta,1/2}:\D \rightarrow \D$ associated to either discretizations is at least of order 2.

\subsection{Some interesting cases of nonholonomic Newmark methods}

The following result shows that, in some particular cases, the nonholonomic Newmark method corresponds with the DLA algorithm, one of the most classical geometric integrators for nonholonomic systems (see \cite{CM2001}).

\begin{proposition}\label{prop:Newmark0isDLA}
	Assume that we have a nonholonomic system defined by a Lagrangian of the type (\ref{lag}) and a distribution ${\mathcal D}$. 
	For any $\alpha\in [0, 1]$ and for $\beta=\beta'=0$, the nonholonomic Newmark method $F^{0, 0, \alpha}_{h}: \D\rightarrow \D$ is equivalent to DLA algorithm with discrete Lagrangian given by
	\begin{multline}\label{alpha:discrete:Lagrangian}
		L_{d}^{sym, \alpha} (q_{k},q_{k+1}) = h\left[  (1-\alpha) L\left(q_{k},\frac{q_{k+1}-q_{k}}{h}\right) + \alpha L\left(q_{k+1},\frac{q_{k+1}-q_{k}}{h}\right)\right]\\
		=h \left(\frac{q_{k+1}-q_k}{h}\right) M\left(\frac{q_{k+1}-q_k}{h}\right)^T
		-h(1-\alpha)V(q_k)-h\alpha V(q_{k+1})
	\end{multline}
	for each $\alpha$ and using either discretization \eqref{alpha:discretization} or \eqref{alpha:discretization-1}.
\end{proposition}

\begin{proof}
In this particular case
\[
\Gamma_L(q_k, v_k)=-M^{-1} \nabla V(q_k)\hbox{    and    } 
Z^{a}(q_k)=M^{-1} \mu^{a} (q_k)
\]
	From the equations of the nonholonomic Newmark method we obtain 
\begin{align*}
	q_k &= q_{k+1} - h v_{k+1} + \frac{h^{2}}{2} \Gamma_{nh}(q_{k+1},v_{k+1}, \lambda'_{k+1}) \\
	q_{k+2} &= q_{k+1} + h v_{k+1} + \frac{h^{2}}{2}\Gamma_{nh}(q_{k+1},v_{k+1}, \lambda_{k+1}) 
	\end{align*}
	adding both equations  
we immediately deduce that
	\begin{equation*}
		\frac{q_{k+2}-2q_{k+1}+q_{k}}{h^{2}} = \Gamma_{L}(q_{k+1},v_{k+1}) + \frac{\lambda_{k+1}+\lambda'_{k+1}}{2}Z(q_{k+1}).
	\end{equation*}
	or equivalently,
		\begin{equation*}
	\frac{q_{k+2}-2q_{k+1}+q_{k}}{h^{2}} = -M^{-1} \nabla V(q_{k+1})+ \frac{\lambda_{k+1}+\lambda'_{k+1}}{2} M^{-1} \mu(q_{k+1}).
	\end{equation*}
	
	These equations are equivalent to DLA integrator with respect to the discrete Lagrangian
	$
		L_{d}^{sym, \alpha} (q_{k},q_{k+1})$ and the constraints (\ref{alpha:discretization}): 
	with the relation between Lagrange multipliers being $$\Lambda = \frac{h(\lambda_{k+1}+\lambda'_{k+1})}{2},$$
	where $\Lambda$ is the Lagrange multiplier appearing in the DLA method \cite{CM2001}.
\end{proof}

Moreover, using the previous method, we can produce new numerical integrators using composition and the adjoint method (see \cite{hairer}, Chapter II.3). If $\Phi_{h}$ is a numerical method then the adjoint method is given by $\Phi_{h}^{*}=(\Phi_{-h})^{-1}$. An example of composition of numerical methods is shown in the next Proposition:
\begin{proposition}\label{compo:old}
	Consider the nonholonomic Newmark method with $\beta=\beta'=0$ and $\alpha=0$ denoted by $F^{0,0, 0}_{h}$, and its adjoint method $(F^{0,0, 0}_{h})^*$. The composition of these two methods
	\begin{equation}\label{composition:Newmark:old}
		\Psi_{h}=(F^{0,0, 0}_{h/2})^* \circ F^{0,0, 0}_{h/2}
	\end{equation}
	generates a second order method, using standard results on composition of adjoint methods.
\end{proposition}

\begin{proof}
The last result holds directly from the results in \cite{hairer}.
\end{proof}

But we can say more, we may prove that $(F^{0,0, 0}_{h})^*=F^{0, 0, 1}_{h}$ and obtain:

\begin{proposition}\label{compo}
	Let $\beta=\beta'=0$. The nonholonomic Newmark methods with $\alpha=0, 1$, denoted by $F^{0,0, 0}_{h}$ and $F^{0, 0, 1}_{h}$ , respectively,  are adjoint methods. Therefore, the composition of these two methods
	\begin{equation}\label{composition:Newmark}
	\Psi_{h}=F^{0, 0, 1}_{h/2} \circ F^{0,0, 0}_{h/2}
	\end{equation}
	generates a second order method, using standard results on composition of adjoint methods.
\end{proposition}

\begin{proof}
	Observe that the method $F^{0,0, 0}_{h}$ is given by the equations
	\begin{align*}
		q_{k+1}& = q_{k} + h v_{k} + \frac{h^{2}}{2} \Gamma_{nh}(q_{k},v_{k}, \lambda_{k}) \\
		q_{k} &= q_{k+1} - h v_{k+1} + \frac{h^{2}}{2}\Gamma_{nh}(q_{k+1},v_{k+1}, \lambda'_{k+1}) \\
		v_{k+1}&\in \D_{q_{k+1}}\\
		0&=	\langle\mu^{a}\left(q_{k}\right), \frac{q_{k+1}-q_{k}}{h}\rangle
	\end{align*}
        It is a straightforward verification that its adjoint method $(F^{0,0, 0}_{h})^{*}$ is given by the same equations except that the last one becomes $0=	\langle\mu^{a}\left(q_{k+1}\right),\frac{q_{k+1}-q_{k}}{h}\rangle$. This means that $(F^{0,0, 0}_{h})^*=F^{0, 0, 1}_{h}$ and the result follows.
\end{proof}

 Consider the Newmark methods with $\beta=\beta'=0$ and a Lagrangian of the type
\begin{equation*}
L(q, \dot{q})=\frac{1}{2}\dot{q} M\dot{q}^T-V(q)\, .
\end{equation*}
 Suppose  that we discretize the constraint space using the parameter $\alpha=0$, i.e.,
	$${\mathcal M}_{q_k,h}^d = \left\{q_{k+1} \ | \ \left\langle \mu^{a}\left(q_{k}\right), \frac{q_{k+1}-q_{k}}{h}\right\rangle = 0\right\}.$$
	Then from the equation
	\[
	q_{k+1} = q_{k} + h v_{k} + \frac{h^{2}}{2} \Gamma_{nh}(q_{k},v_{k}, \lambda_{k}) ,
	\]
	 we explicitly obtain the Lagrange multiplier $\lambda_k$: 
	\begin{equation*}
		\lambda_{k} = \frac{\langle\mu(q_{k}), M^{-1}\nabla V(q_k)\rangle}{\|\mu(q_{k})\|_M^{2}}
	\end{equation*}
	where 
	\[
	\|\mu(q_{k})\|_M = \sqrt{ \mu^a_i(q_k)M^{ij}\mu^b_j(q_k) }\; .
	\]
	In consequence we explicitly derive $q_{k+1}$ as
	\[
	q_{k+1} = q_{k} + h v_{k} + \frac{h^{2}}{2} \left(-M^{-1}\nabla V(q_k)
	+\frac{\langle\mu(q_{k}), M^{-1}\nabla V(q_k)\rangle}{\|\mu(q_{k})\|_M^{2}}M^{-1}\mu(q_k)\right)
	\] 
	
	Then, applying the co-vector $\mu(q_{k+1})$ to the second equation 
	\[
	q_{k} = q_{k+1} - h v_{k+1} + \frac{h^{2}}{2}\Gamma_{nh}(q_{k+1},v_{k+1}, \lambda'_{k+1}) 
	\]
	 we obtain
	\begin{equation*}
		\lambda'_{k+1} = -\frac{2}{h\|\mu(q_{k+1})\|^{2}}\left\langle \mu\left(q_{k+1}\right), \frac{q_{k+1}-q_{k}}{h}\right\rangle + \frac{\langle\mu(q_{k+1}), M^{-1}\nabla V(q_{k+1})\rangle}{\|\mu(q_{k+1})\|^{2}}
	\end{equation*}
	and in consequence we also derive explicitly $v_{k+1}$.  
	\begin{proposition}\label{explict-00}
		The nonholonomic Newmark method $F_h^{0,0,0}: \D\rightarrow \D$ is completely explicit for Lagrangians of the type (\ref{lag}) and constraints of the type $\alpha=0$. 
	\end{proposition}
	\begin{remark}
		In fact Proposition \ref{explict-00} is more general and can be trivially generalized for Lagrangians of mechanical type
		\[
		L(q, \dot{q})=\frac{1}{2}\dot{q} M(q)\dot{q}^T-V(q)
		\]
		where $M(q)$ is a positive definite matrix for all $q\in Q$. 
	\end{remark}

\begin{remark}
	We have introduced in Sections \ref{NhN_section} and \ref{F_betabetaalpha_section} first and second-order nonholonomic Newmark methods depending on the concrete values $(\beta, \beta')$ and the discretization of the  exact discrete  constraint space. 
	As we have seen in Propositions \ref{compo:old} and \ref{compo} we can design new methods preserving the nonholonomic constraints using  the idea of composing methods. In the same way,  we can  produce higher-order nonholonomic integrators using nonholonomic Newmark methods as building blocks  (see \cite{blanes,hairer,Yoshida}) . 
	
	For instance, considering the second-order nonholonomic Newmark method $F_h^{0,0,1/2}: {\mathcal D}\rightarrow {\mathcal D}$ and using the triple jump \cite{hairer}, we obtain a fourth order method $\Psi_h: {\mathcal D}\rightarrow {\mathcal D}$ given by: 
	\[
	\Psi_h=F_{\gamma_1h}^{0,0,1/2}\circ F_{\gamma_2h}^{0,0,1/2}\circ F_{\gamma_1h}^{0,0,1/2}
	\]
	where
	\[
	\gamma_1 = \frac{1}{2-2^{1/3}}, \quad  
	\gamma_2= -\frac{2^{1/3}}{(2-2^{1/3})}\,  .
	\]
	Using similar constructions, we can derive higher-order methods for nonholonomic mechanics with order 6, 8, etc. Thus, other choices of the parameters produce different higher-order numerical methods (see, for instance, \cite{Mclachan}). 
\end{remark}

\section{Examples of nonholonomic systems}

\subsection{Nonholonomic particle}

Consider a particle moving in $Q=\mathbb{R}^3$, with coordinates $q=(x,y,z)$ and velocities $v=( \dot x, \dot y, \dot z)$. The Lagrangian function is given by
\begin{equation*}
L(q,v)= \frac{1}{2}(\dot{x}^{2} + \dot{y}^{2} + \dot{z}^{2})
\end{equation*}
and the motion is subjected to the constraint $\langle\mu(q),v\rangle = \dot{z} - y \dot{x}=0$. We may identify $\mu(q)$ with the vector $(-y,0,1)$. Writing down \eqref{Con-2}, we obtain 
\[\Gamma_{nh}(q,v,\lambda) = \lambda \mu(q).\]

For $\alpha=1/2$ and $\beta=\beta'=0$, the nonholonomic Newmark integrator $F_{h}^{0, 0, 1/2}$ has the following form:
\begin{equation*}
	\begin{cases}
		q_{k+1} = q_k+h v_k+\frac{h^2}{2}\lambda_{k} \mu(q_{k}) \\
		q_k=q_{k+1}-h v_{k+1}+\frac{h^2}{2}\lambda'_{k+1}\mu(q_{k+1}) \\
		z_{k+1}-z_{k}-\frac{y_{k}+y_{k+1}}{2}(x_{k+1}-x_{k})=0 \\
		\dot{z}_{k+1}-y_{k+1}\dot{x}_{k+1}=0.
	\end{cases}
\end{equation*}


Alternatively, the composition of nonholonomic Newmark methods  $F_{h/2}^{0,0,0}$ with $F_{h/2}^{0,0,1}$ as in \eqref{composition:Newmark}, gives the following integrator for the nonholonomic particle:
\begin{equation}\label{nhp:composition:a}
	\begin{cases}
		x_{k+1/2} = x_{k} + \frac{h}{2} \dot{x}_{k} \\
		y_{k+1/2} = y_{k} + \frac{h}{2} \dot{y}_{k} \\
		z_{k+1/2} = z_{k}+\frac{h}{2} \dot{z}_{k}
		\end{cases}
		\quad
	\begin{cases}
		\dot{x}_{k+1/2} = \dot{x}_{k} + (\dot{z}_{k}-y_{k+1/2}\dot{x}_{k})\frac{y_{k+1/2}}{1+y_{k+1/2}^{2}} \\
		\dot{y}_{k+1/2} = \dot{y}_{k} \\
		\dot{z}_{k+1/2} = \dot{z}_{k} - \frac{\dot{z}_{k}-y_{k+1/2}\dot{x}_{k}}{1+y_{k+1/2}^{2}}
\end{cases}
\end{equation}
and
\begin{equation}\label{nhp:composition:b}
	\begin{cases}
		x_{k+1} = \frac{(x_{k+1/2}+\frac{h}{2}\dot{x}_{k+1/2})(y_{k+1/2}^2+1)}{\frac{h}{2}\dot{y}_{k+1/2}y_{k+1/2}+y_{k+1/2}^2+1} \\
		y_{k+1} = y_{k+1/2} + \frac{h}{2} \dot{y}_{k+1} \\
		z_{k+1} = z_{k+1/2} + y_{k+1}(	x_{k+1} - x_{k+1/2})
	\end{cases}
	\quad
	\begin{cases}
		\dot{x}_{k+1} = \frac{\dot{x}_{k+1/2}(y_{k+1/2}^2+1)}{\frac{h}{2}\dot{y}_{k+1/2}y_{k+1/2}+y_{k+1/2}^2+1} \\
		\dot{y}_{k+1} = \dot{y}_{k+1/2} \\
		\dot{z}_{k+1} = y_{k+1}\dot{x}_{k+1}
	\end{cases}
\end{equation}

\subsection{Chaotic nonholonomic particle}

In this example, we study a particle moving on the configuration space $Q=\R^{5}$ with coordinates $q=(x, y_{1}, y_{2}, z_{1}, z_{2})$ and described by the mechanical Lagrangian function \cite{perlmutter06}:
\begin{equation*}
	L(q,\dot{q})= \frac{1}{2}\|\dot{q}\|^{2} - \frac{1}{2}(\|q\|^{2} + z_{1}^{2}z_{2}^{2} + y_{1}^{2}z_{1}^{2} + y_{2}^{2}z_{2}^{2}),
\end{equation*}
where $\|\cdot \|$ denotes the euclidean norm, and subjected to the single constraint $\dot{x} + y_{1} \dot{z}_{1} + y_{2} \dot{z}_{2}=0$ (see~\cite{perlmutter06}).

The motion of the chaotic particle is given by the system of differential equations
\begin{equation*}
	\begin{cases}
		\ddot{x} = -x + \lambda \\
		\ddot{y}_{1} = -y_{1} - y_{1}z_{1}^{2} \\
		\ddot{y}_{2} = -y_{2} - y_{2}z_{2}^{2} \\
		\ddot{z}_{1} = -z_{1} -z_{1}z_{2}^{2} - y_{1}^{2}z_{1} + \lambda y_{1} \\
		\ddot{z}_{2} = -z_{2} -z_{1}^{2}z_{2} - y_{2}^{2}z_{2} + \lambda y_{2} \\
		\dot{x} + y_{1} \dot{z}_{1} + y_{2} \dot{z}_{2}=0.
	\end{cases}
\end{equation*}

For $\alpha=1/2$ and $\beta=\beta'=0$, the nonholonomic Newmark integrator $F_{h}^{0,0,1/2}$ has the following form:
\begin{equation*}
	\begin{cases}
		x_{k+1} = x_{k}+h \dot{x}_{k} + \frac{h^2}{2} (-x_{k} + \lambda_{k}) \\
		y_{1,k+1} = y_{1,k} + h \dot{y}_{1,k} - \frac{h^2}{2} (y_{1,k} + y_{1, k}z_{1, k}^{2}) \\
		y_{2,k+1} = y_{2,k} + h \dot{y}_{2,k} - \frac{h^2}{2} (y_{2,k} + y_{2, k}z_{2, k}^{2}) \\
		z_{1,k+1} = z_{1, k} + h \dot{z}_{1, k} + \frac{h^2}{2} (\lambda_{k}y_{1, k} - z_{1, k}z_{2, k}^{2} - y_{1, k}^{2} z_{1, k} - z_{1, k}) \\
		z_{2,k+1} = z_{2, k} + h \dot{z}_{2, k} + \frac{h^2}{2} (\lambda_{k}y_{2, k} - z_{1, k}^{2}z_{2, k} - y_{2, k}^{2} z_{2, k} - z_{2, k}) \\
		x_{k} = x_{k+1} - h \dot{x}_{k+1} + \frac{h^2}{2} (-x_{k+1} + \lambda'_{k+1}) \\
		y_{1, k} = y_{1, k+1} - h \dot{y}_{1, k+1} - \frac{h^2}{2} (y_{1,k+1} + y_{1, k+1}z_{1, k+1}^{2}) \\
		y_{2,k} = y_{2,k+1} - h \dot{y}_{2,k+1} - \frac{h^2}{2} (y_{2,k+1} + y_{2, k+1}z_{2, k+1}^{2}) \\
		z_{1,k} = z_{1, k+1} - h \dot{z}_{1, k+1} + \frac{h^2}{2} (\lambda'_{k+1}y_{1, k+1} - z_{1, k+1}z_{2, k+1}^{2} - y_{1, k+1}^{2} z_{1, k+1} - z_{1, k+1}) \\
		z_{2,k} = z_{2, k+1} - h \dot{z}_{2, k+1} + \frac{h^2}{2} (\lambda'_{k+1}y_{2, k+1} - z_{1, k+1}^{2}z_{2, k+1} - y_{2, k+1}^{2} z_{2, k+1} - z_{2, k+1}) \\
		x_{k+1}-x_{k}+\frac{y_{1,k}+y_{1,k+1}}{2}(z_{1,k+1}-z_{1, k}) + \frac{y_{2,k}+y_{2,k+1}}{2}(z_{2,k+1}-z_{2, k})=0 \\
		\dot{x}_{k+1} + y_{1,k+1} \dot{z}_{1,k+1} + y_{2,k+1} \dot{z}_{2,k+1}=0.
\end{cases}
\end{equation*}

\subsection{Pendulum-driven CVT system}

This example in $Q=\R^{3}$ is a nonholonomic continuous variable transmission (CVT) system determined by an independent Hamiltonian subsystem called the driver system \cite{modin}. We will denote the coordinates in $\R^3$ by $(x,y,\xi)$ and, then, the Lagrangian function is
$$L(x,y,\xi,\dot{x},\dot{y},\dot{\xi})= \frac{1}{2}\left(\sum_{i=1}^{2} \dot{q}_{i}^{2} + \kappa_{i}q_{i}^{2}  \right) + l(\xi, \dot{\xi}),$$
where $(q_{1},q_{2},\dot{q}_{1},\dot{q}_{2})=(x, y, \dot{x}, \dot{y})$ and $l(\xi, \dot{\xi}) = \frac{1}{2}\dot{\xi}^{2} - V(\xi)$. The nonholonomic constraint is of the form
$$\dot{y} + f(\xi)\dot{x}=0.$$

The motion of this family of systems is given by the equations
\begin{equation*}
	\begin{cases}
		\ddot{x} = \kappa_{1}x + \lambda f(\xi) \\
		\ddot{y} = \kappa_{2}y + \lambda \\
		\ddot{\xi} = - V'(\xi) \\
		\dot{y} + f(\xi)\dot{x}=0
	\end{cases}
\end{equation*}
where the Lagrange multiplier may be computed to be of the form
\begin{equation*}
	\lambda = - \frac{f'(\xi)\dot{\xi}\dot{x} + \kappa_{1} f(\xi)x + \kappa_{2}y}{1 + f^{2} (\xi)}.
\end{equation*}

From now on, consider the following potential and constraint functions and constants
\begin{equation*}
	V(\xi) = \cos(\xi) - \frac{ \epsilon \sin(2 \xi)}{2}, \quad f(\xi)=\sin(\xi), \quad \kappa_{1}=\kappa_{2}=-1.
\end{equation*}
This example has the property that for $\epsilon \neq 0$, the system is no longer integrable reversible and so, good long time behaviour observed in most nonholonomic integrators is lost in this case {(see \cite{modin})}.

For $\alpha=1/2$ and $\beta=\beta'=0$, the Newmark nonholonomic integrator $F_{h}^{0,0,1/2}$ has the following form
\begin{equation*}
\begin{cases}
	x_{k+1} = x_{k}+h \dot{x}_{k} + \frac{h^2}{2}( -x_{k} + \sin(\xi_{k}) \lambda_{k}) \\
	y_{k+1} = y_{k} + h \dot{y}_{k} + \frac{h^2}{2}( -y_{k} + \lambda_{k}) \\
	\xi_{k+1} = \xi_{k} + h \dot{\xi}_{k} + \frac{h^2}{2} (\sin(\xi_{k})+\epsilon \cos(2 \xi_{k})) \\
	x_{k} = x_{k+1} - h \dot{x}_{k+1} + \frac{h^2}{2}( -x_{k+1} + \sin(\xi_{k+1}) \lambda'_{k+1}) \\
	y_{k} = y_{k+1} - h \dot{y}_{k+1} + \frac{h^2}{2} (-y_{k+1} + \lambda'_{k+1}) \\
	\xi_{k} = \xi_{k+1} - h \dot{\xi}_{k+1} + \frac{h^2}{2} (\sin(\xi_{k+1})+\epsilon \cos(2 \xi_{k+1})) \\
	y_{k+1}-y_{k} + \sin\left(\frac{\xi_{k} + \xi_{k+1}}{2}\right)(x_{k+1}-x_{k})=0 \\
	\dot{y}_{k+1}+\sin(\xi_{k+1})\dot{x}_{k+1}=0.
\end{cases}
\end{equation*}

This system of difference equations has also an explicit solution but we omit it here since it is too long to be meaningful.

\section{Numerical results}\label{numerical_results}

\subsection{Nonholonomic particle}
Figure~\ref{nhparticle_energy_drift} shows the energy drift for the nonholonomic Newmark method with $\beta=\beta'=0$ (which is  equivalent to the DLA method by Proposition~\ref{prop:Newmark0isDLA}), a Runge-Kutta 4th order method and the composition method $\Psi_h$ in Proposition~\ref{compo}. Here $T=100$, $h=.2$, and the initial conditions are $q_0=(1,1,-1)$, $v_0=(1,-1,1)$  with energy $1.5$. In this example, as well as in all the following ones, we used $\alpha=1/2$, except of course for $\Psi_h$. 

We can see that in this simple system, the Runge-Kutta method shows a better energy behaviour than the other two.

\begin{figure}[htb!]
  \begin{center}
    \includegraphics[scale=.5]{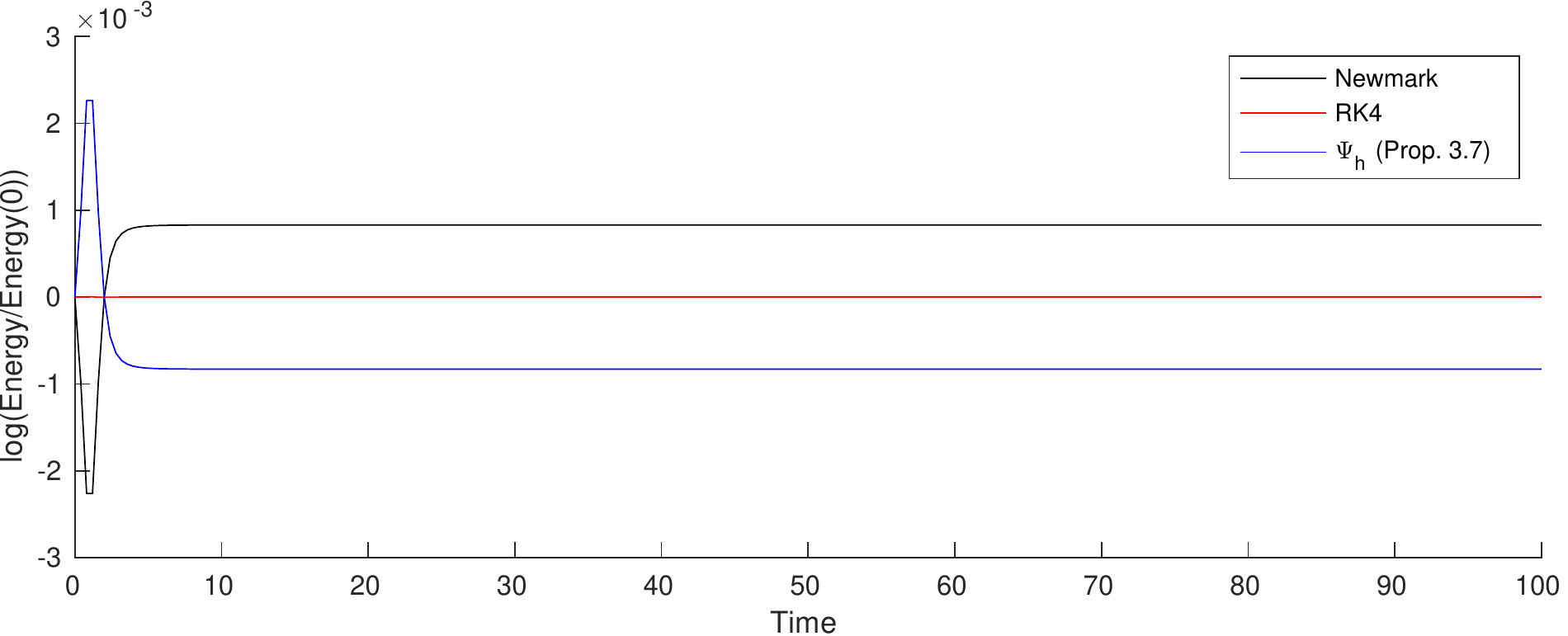}
    \caption{Energy drift for the nonholonomic particle. The line for RK4 is not exactly zero, but it settles around $5\cdot 10^{-7}$.}
    \label{nhparticle_energy_drift}
  \end{center}
\end{figure}

\subsection{Chaotic nonholonomic particle}
Figure~\ref{chaotic_energy_drift} shows the energy drift for the nonholonomic Newmark method with $\beta=\beta'=0$ (DLA), the nonholonomic Newmark method with $\beta=\beta'=.1$, a Runge-Kutta 4th order method and the composition method $\Psi_h$ in Proposition~\ref{compo}. Here we used $T=1000$, $h=.2$, and initial conditions $q_0=(1,0,1,-1,-1)$, $v_0=(0.05,   0.5,   -0.5,   -0.1,   -0.05)$, with energy $3.2575$ approximately.

For this example, the methods we propose here outperform Runge-Kutta in energy behaviour.

\begin{figure}[htb!]
  \begin{center}
    \includegraphics[scale=.5]{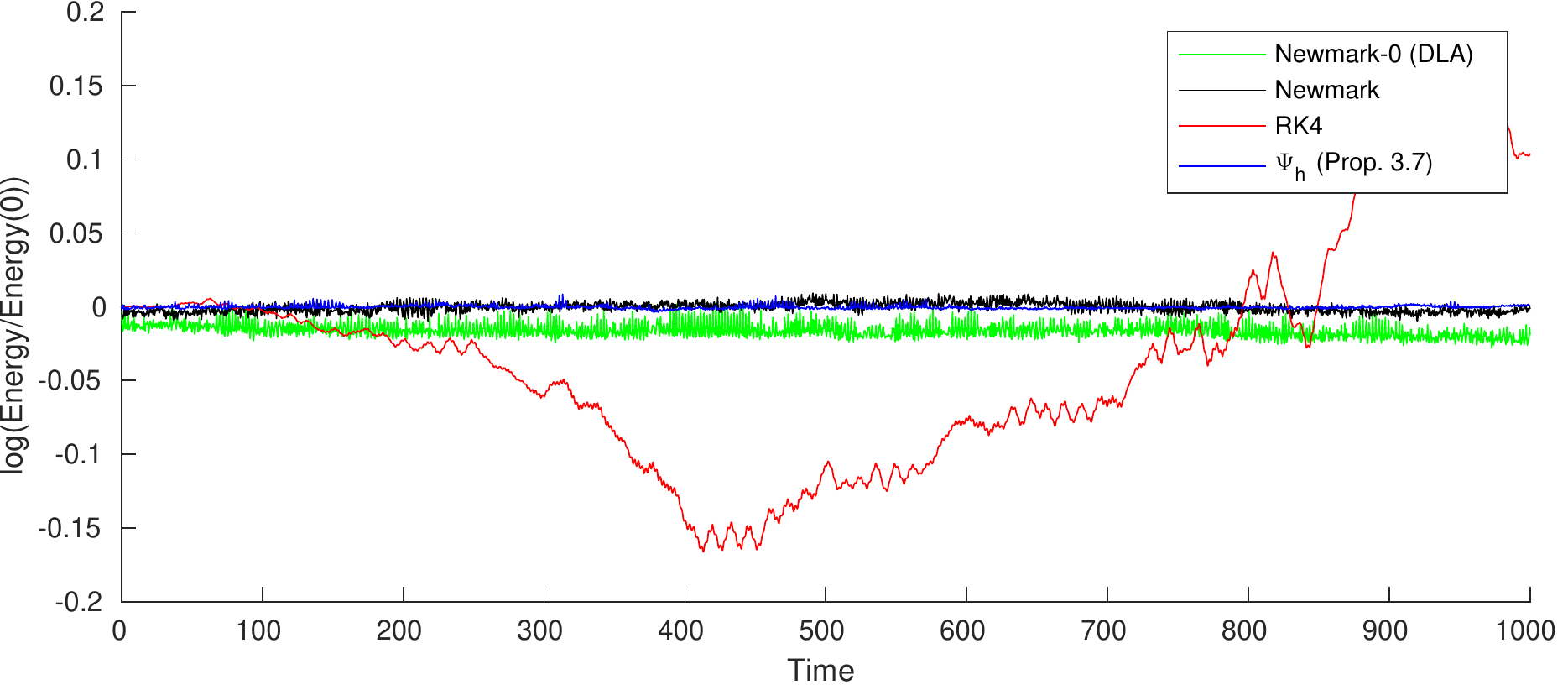}
    \caption{Energy drift for the chaotic nonholonomic particle.}
    \label{chaotic_energy_drift}
  \end{center}
\end{figure}

We also explore 100 random initial conditions for this example, all of them having the same energy value of $1.535$. We used $\beta=\beta'=.1$, $T=10000$ and $h=.2$. The method used is the composition method in Proposition~\ref{compo}. In Figure~\ref{figmanytrajectories} we plot the energy drift for each trajectory and the variance of the energy drift as in \cite{perlmutter06}.

\begin{figure}[htb!]
  \begin{center}
    \includegraphics[scale=.49]{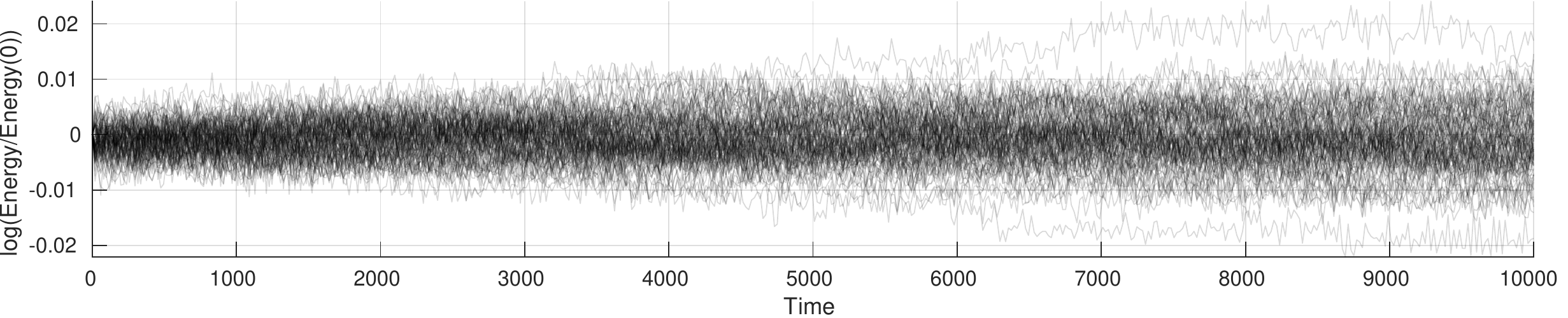}
    \includegraphics[scale=.49]{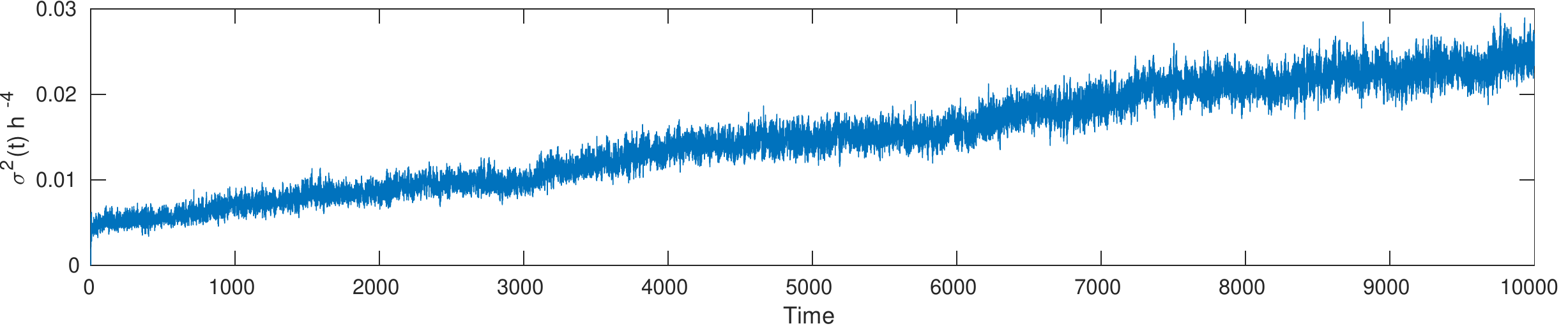}
    \caption{Energy drift and variance for 100 random trajectories of the chaotic nonholonomic particle, with $\beta=\beta'=.1$.}
    \label{figmanytrajectories}
  \end{center}
\end{figure}

\subsection{Pendulum-driven CVT}
We first consider the case  $\epsilon=0$. In Figure~\ref{CVT0} we show the energy drift for the nonholonomic Newmark method with $\beta=\beta'=0$, and with $\beta=\beta'=.1$, 4th-order Runge-Kutta, and the composition method $\Psi_h$. Here $T=400$, $h=.2$, and the initial conditions are $q_0=(1,0,-2)$, $v_0\approx (-0.4481,-0.4075,0.1)$, with an approximate energy of $0.2723$.
\begin{figure}[htb!]
	\begin{center}
		\includegraphics[scale=.5]{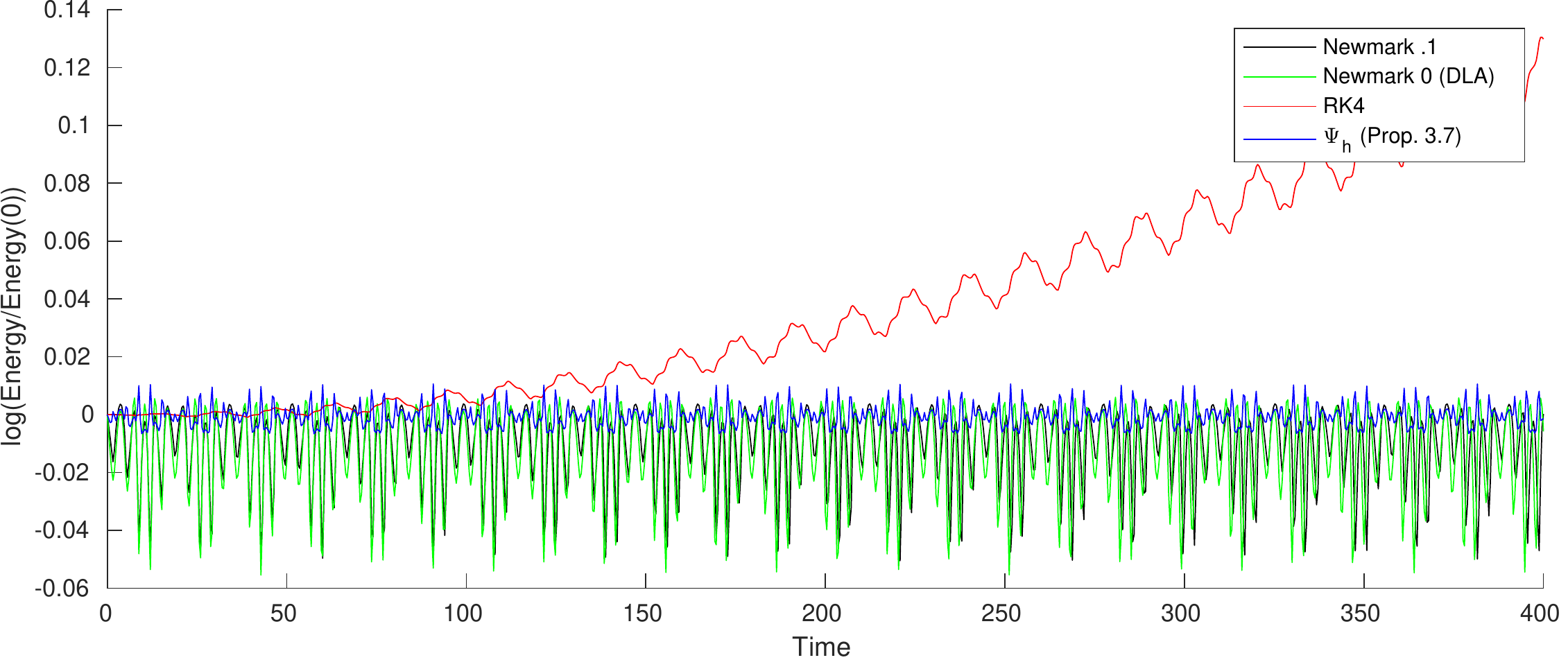}
		\caption{Energy drift for the pendulum-driven CVT, $\epsilon=0$.}
		\label{CVT0}
	\end{center}
\end{figure}

For the case $\epsilon=0.1$ in Figure~\ref{CVT1}, we compare the nonholonomic Newmark method with $\beta=\beta'=0$ and the composition methods $\Psi_h$ and $\Psi_{2h}$. The latter was included because the computational cost for each step of $\Psi_h$ is twice that of the nonholonomic Newmark method; therefore $\Psi_{2h}$ has a global computational cost comparable to the nonholonomic Newmark method. Here $h=.05$, $T=1500$, and the initial conditions are the same as the ones used in \cite{modin}, which are $q_0=(1,1,0)$, $v_0\approx (0,0,2.82842712)$, now the energy being exactly $6.0$.

\begin{figure}[htb!]
	\begin{center}
		\includegraphics[scale=.5]{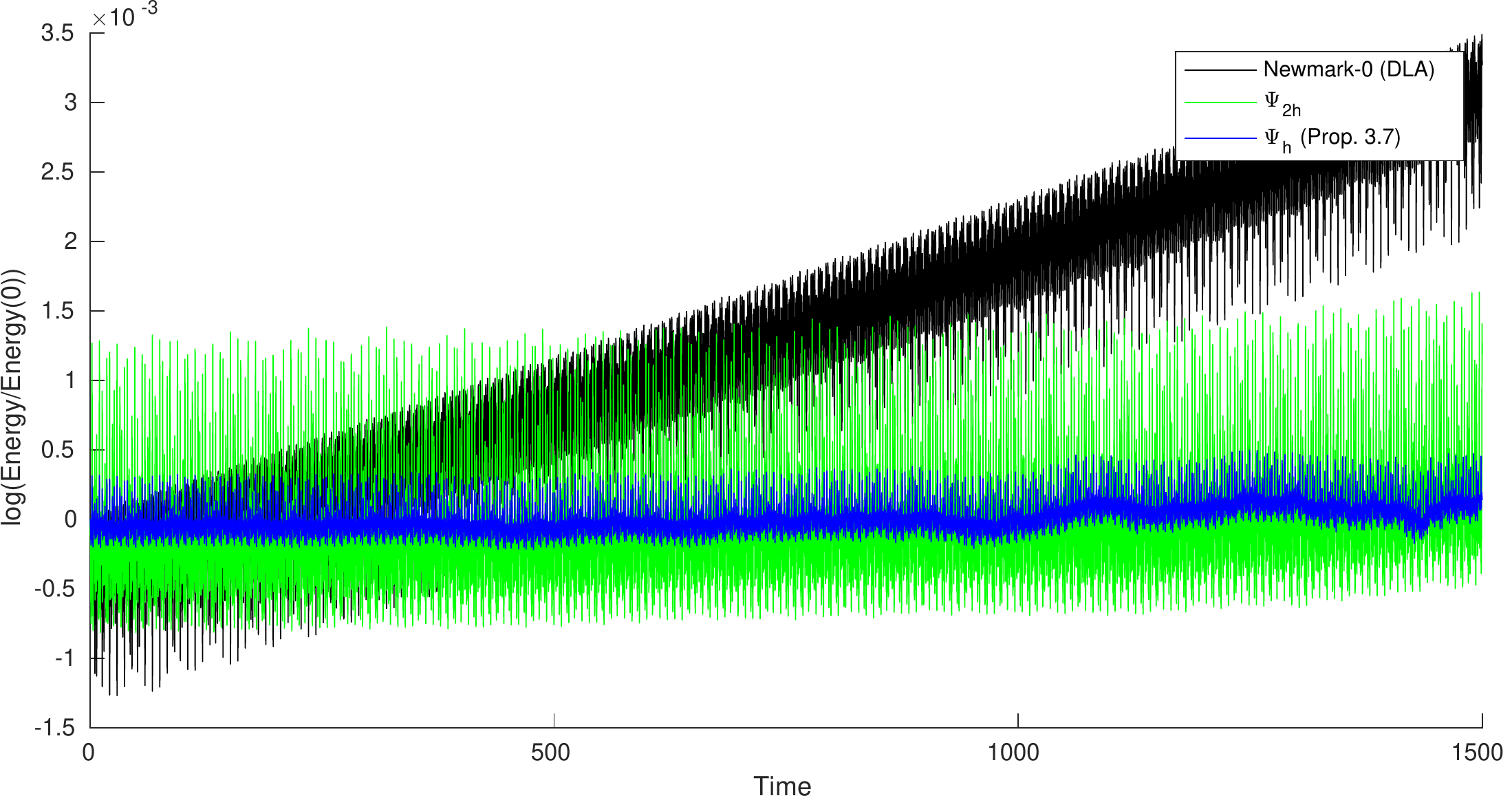}
		\caption{Energy drift for the pendulum-driven CVT, $\epsilon=0.1$.}
		\label{CVT1}
	\end{center}
\end{figure}

As expected, we observe that the Newmark method, being equivalent to DLA method, is no longer able to preserve energy as it had been already pointed out in \cite{modin}. However, the composition of the two Newmark methods in \eqref{composition:Newmark} shows nearly preservation of energy. At the moment, we have no explanation for this good behaviour.

\section{Future work}

In a future paper, we will study the extension of the nonholonomic Newmark method to non-linear spaces, that is, in general differentiable manifolds. In particular, if $Q=G$ is a Lie group we can derive a nonholonomic Lie-Newmark method (see \cite{Krysl}) where we assume that we have a retraction map $R: \mathfrak{g}\rightarrow G $ (for instance the Lie group exponential map) and we identify by left (right)-trivialization $TG\equiv G\times{\mathfrak g}$ with left (respectively, right)-trivialized coordinates $(g, \xi)$. Therefore if  $g_k\in G$ and $\xi_k\in g_k^{-1}{\mathcal D}_{g_k}\subseteq  {\mathfrak g}$ then
\begin{align*}
	g_k^{-1}g_{k+1}&=R( h\xi_k+\frac{h^2}{2}\Gamma_{nh}(g_k, \xi_k, \lambda_k))\\
	\frac{\xi_{k+1}-\xi_k}{h}&=\frac{1}{2}\Gamma_{nh}(g_k, \xi_k, \lambda_k)+\frac{1}{2}\Gamma_{nh}(g_{k+1}, \xi_{k+1}, \lambda'_{k+1})\\
	g_{k+1}&\in {\mathcal M}^d_{g_k,h}\\
	\xi_{k+1}&\in g_{k+1}^{-1}{\mathcal D}_{g_{k+1}} 
\end{align*}
where we have also identified $TTG\equiv G\times {\mathfrak g}\times  {\mathfrak g}\times  {\mathfrak g}$ and then $\Gamma_{nh}(g_k, \xi_k, \lambda_k)\in {\mathfrak g}$ and ${\mathcal M}^d_{g_k,h}$ is a discretization of the exact discrete constraint space. 

\section{Acknowledgements}
A. Anahory Simoes and D. Mart{\'\i}n de Diego acknowledge financial support from the Spanish Ministry of Science and Innovation, under grant PID2019-106715GB-C21 and  the ``Severo Ochoa Programme for Centres of Excellence'' in R\&D from CSIC (CEX2019-000904-S).
A. Anahory Simoes is also supported by a 2020 Leonardo Grant
for Researchers and Cultural Creators, BBVA Foundation.
 S.\ Ferraro acknowledges financial support from PICT 2019-00196, FONCyT, Argentina, and PGI 2018, UNS.  Juan Carlos Marrero acknowledges financial support from
the Spanish Ministry of Science and Innovation under grant PGC2018-098265-B-C32.

\bibliography{references}

\end{document}